\documentclass[11pt]{article}
\title{On maximal $S$-free sets and the Helly number\\ for the family of $S$-convex sets}
\author{Gennadiy Averkov\footnote{Institute for Mathematical Optimization, Faculty of Mathematics, University of Magdeburg, 39106 Magdeburg, Germany (email: averkov@math.uni-magdeburg.de)}}
\usepackage{amsmath,amsthm,amssymb,enumerate,empheq}
\usepackage[active]{srcltx}
\usepackage{framed}
\newcommand{\real}{\mathbb{R}}
\newcommand{\natur}{\mathbb{N}}
\newcommand{\intr}{\operatorname{int}}
\newcommand{\sprod}[2]{\left<{#1}\,,\,{#2}\right>}
\newcommand{\dotvar}{\,\cdot\,}
\newcommand{\cF}{\mathcal{F}}

\newcommand{\integer}{\mathbb{Z}}
\newcommand{\setcond}[2]{\left\{ #1 \,:\, #2 \right\}}
\newcommand{\bigsetcond}[2]{\bigl\{ #1 \,:\, #2 \bigr\}}

\newcommand{\cl}{\operatorname{cl}}

\newcommand{\thmheader}[1]{{\upshape(#1).}}
\newcommand{\relint}{\operatorname{relint}}
\newcommand{\relintr}{\relint}
\newcommand{\eps}{\varepsilon}
\newcommand{\conv}{\operatorname{conv}}

\newcommand{\bd}{\operatorname{bd}}
\newcommand{\rank}{\operatorname{rank}}
\newcommand{\aff}{\operatorname{aff}}
\newcommand{\sizetwocol}[2]{\left(\begin{smallmatrix} {#1} \\ {#2} \end{smallmatrix}\right)}
\newcommand{\term}[1]{\emph{#1}}

\newtheorem{theorem}{Theorem}[section]

\newtheorem{lemma}[theorem]{Lemma}
\newtheorem{proposition}[theorem]{Proposition}
\newtheorem{definition}[theorem]{Definition}

\theoremstyle{definition}

\newtheorem{remark}[theorem]{Remark}

\numberwithin{equation}{section}

\begin{document}

\maketitle

\begin{abstract}
	We study two combinatorial parameters, which we denote by $f(S)$ and $h(S)$, associated to an arbitrary set $S \subseteq \real^d$, where $d \in \natur$. In the nondegenerate situation, $f(S)$ is the largest possible number of facets of a $d$-dimensional polyhedron $L$ such that the interior of $L$ is disjoint with $S$ and $L$ is inclusion-maximal with respect to this property. The parameter $h(S)$ is the Helly number of the family of all sets that can be given as the intersection of $S$ with a convex subset of $\real^d$.  We obtain the inequality $f(S) \le h(S)$ for an arbitrary $S$ and the equality $f(S)=h(S)$ for  every discrete $S$. Furthermore, motivated by research in integer and mixed-integer optimization, we show that $2^d$ is the sharp upper bound on $f(S)$ in the case $S = (\integer^d \times \real^n) \cap C$, where $n \ge 0$ and $C \subseteq \real^{d+n}$  is convex. The presented material generalizes and unifies results of various authors, including the result $h(\integer^d) = 2^d$ of Doignon, the related result $f(\integer^d)=2^d$ of Lov\'asz and the inequality $f(\integer^d \cap C) \le 2^d$, which has recently been proved for every convex set $C \subseteq \real^d$ by Dey \& Mor\'an.
\end{abstract}

\newtheoremstyle{itsemicolon}{}{}{\mdseries\rmfamily}{}{\itshape}{:}{ }{}

\newtheoremstyle{itdot}{}{}{\mdseries\rmfamily}{}{\itshape}{:}{ }{}

\theoremstyle{itdot}

\newtheorem*{msc*}{2010 Mathematics Subject Classification}

\begin{msc*}
	Primary 90C11;  Secondary 52A01, 52C07, 90C10
\end{msc*}

\newtheorem*{keywords*}{Keywords}

\begin{keywords*}
	cutting plane; Doignon's theorem; Helly's theorem; Helly number; intersection cut; lattice-free set; $S$-convex set; $S$-free set
\end{keywords*}

\section{Introduction}

For background information from convex geometry, theory of polyhedra and geometry of numbers we refer to \cite{MR2335496,MR893813,MR1242973,MR0274683,MR1216521,MR874114}. Let $\natur:= \{1,2,\ldots\}$ and $\natur_0 := \{0,1,2,\ldots\}$. For $m \in \natur$ we use the notation  $[m]:=\{1,\ldots,m\}$. Let $d \in \natur$ and $S \subseteq \real^d$. The origin of linear spaces (such as $\real^d$) is denoted by $o$. In this manuscript we study $S$-free sets and maximal $S$-free sets, defined as follows.

\begin{definition} \thmheader{$S$-free set and maximal $S$-free set}
	A set $L \subseteq \real^d$ is said to be \term{$S$-free} if $L$ is closed, convex and the interior of $L$ is disjoint with $S$. An $S$-free set $L$ in $\real^d$ is said to be \term{maximal $S$-free} if there exists no $S$-free set properly containing $L$. 
\end{definition}

We consider the question whether for a given $S$ there exists a bound $k \in \natur_0$ such that every maximal $S$-free set is a polyhedron with at most $k$ facets. If such a bound $k$ exists, we are interested in finding an appropriate $k$ explictly. Our question can be expressed in terms of the following functional.

\begin{definition} \thmheader{The functional $f$ -- facet complexity of maximal $S$-free sets}
	Let $S \subseteq \real^d$. If there exists $k \in \natur_0$ such that every $d$-dimensional maximal $S$-free set is a polyhedron with at most $k$ facets, we define $f(S)$ to be the minimal $k$ as above. If there exist no $d$-dimensional maximal $S$-free sets (e.g., for $S=\real^d$) we let $f(S):=-\infty$. If there exist maximal $S$-free sets which are not polyhedra or maximal $S$-free polyhedra with arbitrarily large number of facets we define $f(S):=+\infty$.
\end{definition}

Thus, in the qualitative form our question is about conditions which ensure $f(S)<+\infty$. Quantitatively, we are interested in bounds on $f(S)$. More specifically, we are interested in evaluating or estimating $f(S)$ for certain structured sets $S$ that play a role in optimization. With a view toward applications in the cutting-plane theory from integer and mixed-integer optimization it is desirable to have upper bounds on $f((\integer^d \times \real^n) \cap C)$, where $d\in \natur, n \in \natur_0$ and $C \subseteq \real^{d+n}$ is convex. See also \cite{MR2480507,MR2855866,MR0290793,conforti-et-al-survey-2011,MR2968262,MR2795192,MR2828002} for information on application of $S$-free sets for generation of cutting planes. The main topic of this manuscript is the study of the relationship between $f(S)$ and the Helly number $h(S)$ for the family of of $S$-convex sets.  The notions of $S$-convex set and the Helly number are introduced below.

\begin{definition} \label{S:conv:def} \thmheader{$S$-convex set}
	Given $S \subseteq \real^d$, a set $A \subseteq \real^d$ is called \emph{$S$-convex} if $A = S \cap C$ for some convex set $C \subseteq \real^d$. 
\end{definition}

In order to avoid possible ambiguities we point out that the literature contains a number of different generalizations of the notion of convexity (see, for example, the monographs \cite{MR1439963} and \cite{MR1234493}). In \cite[\S\S1.9]{MR1234493} the $S$-convexity as introduced here is called the relative convexity of $S$. For $S=\real^d$ the $S$-convexity is reduced to the standard convexity. The notion of $S$-convexity in the case $S=\integer^d$ was considered by various authors in different contexts (see, for example, \cite{MR1963693,MR0387090,MR1837217,MR2160045,MR1105949}). To the best of author's knowledge, the study of $S$-convexity in the case $S=\real^d \times \integer^n$ has been initiated in \cite{arXiv:1002.0948}.

\begin{definition} \thmheader{The functional $h$ -- Helly number}
	Given a nonempty family $\cF$ of sets, the \emph{Helly number} $h(\cF)$ of $\cF$ is defined as follows. For $\cF=\{\emptyset\}$ let $h(\cF) := 0$.
	If $\cF \ne \{\emptyset\}$ and there exists $k \in \natur$ such that
	\begin{align} \label{helly:cond}
		&F_1 \cap \cdots \cap F_m = \emptyset & &\Longrightarrow & & \exists i_1,\ldots,i_k \in [m] \ : \ F_{i_1}  \cap \cdots \cap F_{i_k} = \emptyset
	\end{align}
	for all $F_1, \ldots, F_m \in \cF$ ($m \in \natur$), then we define $h(\cF)$ to be the minimal $k$ as above. In all other cases we let $h(\cF):=+\infty$. For $S \subseteq \real^d$ we use the notation
	\begin{equation} \label{h:S:def}
		h(S) := h\bigl( \bigsetcond{S \cap C}{C \subseteq \real^d \ \text{is convex\,} }\bigr).
	\end{equation}
	That is, $h(S)$ is the  Helly number of the family of all $S$-convex sets.
\end{definition}

The functional $h$ given by \eqref{h:S:def} has several interpretations in terms of optimization; see \cite[Proposition~1.2]{arXiv:1002.0948}.

The main results of this manuscript, which we formulate in the next section, can be split into two groups: theorems about $f(S)$ and $h(S)$ for general sets $S$ having no particular (global) structure and, on the other hand, theorems providing bounds on $f(S)$ and $h(S)$ for structured sets $S$, whose structure is related to integer and mixed-integer programming.

For general sets $S \subseteq \real^d$, we derive the inequality $f(S) \le h(S)$ and the equality $f(S) = h(S)$ in the case that $S$ is discrete (see Theorem~\ref{main:thm}). We also relate $f(S)$ and $h(S)$ to the sequence of values  $f(S_t)$ and $h(S_t)$, respectively, in the case of set sequences $(S_t)_{t \in \natur}$ satisfying $S_t \subseteq S_{t+1} \forall t \in \natur$ and $S = \bigcup_{t \in \natur} S_t$ (see Theorem~\ref{liminf:thm}).

We show that the above results yield short and unified proofs of the equality $h(\integer^d)=2^d$ of Doignon \cite{MR0387090} (see also \cite{Bell77,MR0452678,Todd1977}), the equality $f(\integer^d) = 2^d$ of Lov\'asz \cite{arXiv:1110.1014} and the inequality $f(\integer^d \cap C) \le 2^d$ for every convex set $C \subseteq \real^d$ recently proved by Dey \& Mor\'an \cite{MR2801234}. Note that various special cases of $f(\integer^d \cap C) \le 2^d$ were derived and used as a tool for research in integer optimization in \cite[Theorem~1.1]{MR2724071}, \cite[Theorem~2]{MR2600658}, \cite[Proposition~31]{MR2735936} and  \cite[Theorem~3.2]{MR2795192}. In \cite{MR2724071} the set $C$ is assumed to be an arbitrary affine space, while in \cite{MR2600658} the set $C$ is an arbitrary rational polyhedron.

We prove that $2^d$ is the tight upper bound on $f((\integer^d \times \real^n) \cap C)$ for $d \in \natur, n \in \natur_0$ and an arbitrary convex set $C \subseteq \real^{d+n}$ (see Theorem~\ref{mixed:integer:f:thm}). The latter is a generalization of the mentioned result of Dey \& Mor\'an to the mixed-integer setting. Observe that the set $\integer^d \times \real^n$ is a Minkowski sum of the lattice $\integer^d \times \{o\}$ of rank $d$ and the linear space $\{o\} \times \real^n$ of dimension $n$. By this we are motivated to study $f(S)$ and $h(S)$ for structured sets $S$, whose structure can be expressed in terms of Minkowski sums and/or lattices. In fact, the upper bound on $f((\integer^d \times \real^n) \cap C)$ will be deduced as a direct consequence of results for such sets $S$ (see Theorems~\ref{part:mink:sum} and \ref{part:lattices}).

We do not systematically address the interesting question of characterizing maximal $S$-free sets. Some information on this question can be found in \cite{arXiv:1110.1014,MR2724071,MR2600658,MR2801234,MR1114315}. 

The manuscript is organized as follows. In Section~\ref{sect:results} we formulate the main results. Section~\ref{sect:background} provides the necessary background material. In Section~\ref{sect:general} we prove results for general sets $S$ and show how to derive the known relations $h(\integer^d) = f(\integer^d) = 2^d$ and $f(\integer^d \cap C) \le 2^d$ (for an arbitary convex $C \subseteq \real^d$) as a consequence. In Section~\ref{sect:structured} we give upper bounds on $f(S)$ and $h(S)$ for structured sets $S$.

\section{Main results} \label{sect:results}

\subsection*{Results on $f(S)$ and $h(S)$ for general sets $S$}

A set $S\subseteq \real^d$ is said to be \term{discrete} if every bounded subset of $S$ is finite. 

\begin{theorem} \label{main:thm} \thmheader{Relation between $f$ and $h$}
	Let $S \subseteq \real^d$. Then the following statements hold:
	\begin{enumerate}[I.]
		\item \label{part:f<=h} $f(S) \le h(S)$.
		\item \label{part:f=h} If $S$ is discrete, one has $f(S)=h(S)$.
	\end{enumerate}
\end{theorem}

The following continuity-type result can be used to bound $f(S)$ or $h(S)$ for a `complicated set' $S$ by approximating $S$ with `simple sets' $S_t$ ($t \in \natur$). 

\begin{theorem} \thmheader{$\liminf$ theorem} \label{liminf:thm} Let $S \subseteq \real^d$. Let $(S_t)_{t=1}^{+\infty}$ be a sequence of sets satisfying $S_t \subseteq S_{t+1} \subseteq \real^d \ \forall t \in \natur$ and $S = \bigcup_{t=1}^{+\infty} S_t$. Then
\begin{align}
	h(S) & \le \liminf_{t \rightarrow +\infty} h(S_t), \label{h:liminf:ineq} \\
	f(S) & \le \liminf_{t \rightarrow +\infty} f(S_t). \label{f:liminf:ineq}
\end{align}
\end{theorem}

\subsection*{Upper bounds on $f(S)$ and $h(S)$ for structured sets $S$}

For $A, B \subseteq \real^d$ the \emph{Minkowski sum} $A+B$ and \emph{Minkowski difference} $A-B$ of $A$ and $B$ are defined by $A \pm B := \setcond{a \pm b}{a \in A, \ b \in B}$.

\begin{theorem} \label{part:mink:sum} \thmheader{On adding a convex set}
	Let $S \subseteq \real^d$ be closed. Let $C \subseteq \real^d$ be nonempty and convex. Then
		\begin{align}
			h(C+S)  & \le  (\dim(C)+1) h(S), \label{h:mink:sum:ineq} \\
			f(C+S)  & \le  f(S). \label{f:mink:sum:ineq}
		\end{align}
\end{theorem}

A set $\Lambda \subseteq \real^d$ is said to be a \term{lattice} if $\Lambda$ is a discrete subgroup of the additive group $\real^d$. Every lattice $\Lambda \subseteq \real^d$, $\Lambda \ne \{o\}$ can be given by 
\[
	\Lambda = \setcond{t_1 x_1 +  \cdots + t_r x_r}{t_1,\ldots,t_r \in \integer},
\]
where $r \in \{1,\ldots,d\}$ and $x_1,\ldots,x_r \in \real^d$ are linearly independent vectors. The value $r$ above is called the rank of the lattice $\Lambda$ and is denoted by $\rank(\Lambda)$. For $\Lambda=\{o\}$ we define $\rank(\Lambda)=0$.

\begin{theorem} \label{part:lattices} Let $C,D \subseteq \real^d$ be nonempty convex sets and let $\Lambda \subseteq \real^d$ be a lattice. Then
	\begin{align}
		h( (C+ \Lambda) \cap D) & \le (\dim(C)+1) 2^{\rank(\Lambda)}, \label{h:mink:sum:lat:ineq} \\
		f( (C+ \Lambda) \cap D) & \le 2^{\rank(\Lambda)}. \label{f:mink:sum:lat:ineq}
	\end{align}
\end{theorem}

\begin{theorem} \label{mixed:integer:f:thm}
	Let $d \in \natur,n \in \natur_0$ and let $C \subseteq \real^{d+n}$ be convex. Then 
	\begin{equation} \label{mixed:integer:f:ineq}
		f((\integer^d \times \real^n) \cap C) \le 2^d.
	\end{equation}
	Furthermore, \eqref{mixed:integer:f:ineq} is attained with equality for $C=\real^{d+n}$.
\end{theorem}

If $C$ in \eqref{f:mink:sum:lat:ineq} is a linear space which does not contain nonzero vectors of $\Lambda$, inequalities \eqref{f:mink:sum:lat:ineq} and \eqref{mixed:integer:f:ineq} are  equivalent. In this special case, \eqref{f:mink:sum:lat:ineq} is a  `coordinate-free' version of \eqref{mixed:integer:f:ineq}.

An analog of Theorem~\ref{mixed:integer:f:thm} for the functional $h$ is provided by \cite{arXiv:1002.0948}. By the main result of \cite{arXiv:1002.0948},
\begin{equation} \label{mixed-int:helly:ineq}
	h((\integer^d \times \real^n) \cap C) \le (n+1)2^d
\end{equation}
 for every convex set $C \subseteq \real^{d+n}$, with equality attained for $C=\real^{d+n}$.

We emphasize that the assertions for $f$ are the main parts of  Theorems~\ref{liminf:thm}--\ref{part:lattices}, while the assertions for $h$ are given as  a complement, in order to highlight the analogy between $f$ and $h$. Inequality \eqref{h:liminf:ineq}  will be derived as a consequence of basic properties of $h$, while \eqref{h:mink:sum:ineq} and \eqref{h:mink:sum:lat:ineq} will be shown to follow directly from the results of \cite{arXiv:1002.0948}. The assertions for $f$ in Theorems~\ref{main:thm} and \ref{liminf:thm} are proved using a compactness argument. The proof of \eqref{f:mink:sum:ineq} relies on basic convex geometry. The proof of \eqref{f:mink:sum:lat:ineq} and \eqref{mixed:integer:f:ineq} involves the so-called parity argument, which is presented in the following section.

\section{Background material} \label{sect:background}

Throughout the manuscript we use the following notation. The cardinality of a set $X$ is denoted by $|X|$. The standard scalar product of $\real^d$ is denoted by $\sprod{\dotvar}{\dotvar}$. By $\aff, \bd, \cl, \conv, \intr$ and $\relintr$ we denote the affine hull, boundary, closure, convex hull, interior (of a set) and relative interior (of a convex set), respectively.

\subsection*{The role of the parity argument} \label{sect:parity:arg}

Let $d \in \natur$. The following implication is usually referred to as the \emph{parity argument}: if $X \subseteq \integer^d$ and $|X|>2^d$, then there exist $x, y \in X$ with $x \ne y$ and $\frac{1}{2}(x+y) \in \integer^d$. This implication is obtained by comparing the elements of $X$ modulo $2 \integer^d$. Note that the parity argument is the key step in the existing proofs of $f(\integer^d) = 2^d$ and $h(\integer^d)=2^d$.

We illustrate how to use the parity argument by sketching a proof of the inequality $f(\integer^d \cap C) \le 2^d$ in the case that $C$ is a \emph{bounded}, convex subset of $\real^d$. Let $S:=\integer^d \cap C$. If $S = \emptyset$, there is nothing to prove. Thus, we assume $S \ne \emptyset$. Let $L$ be an arbitrary $d$-dimensional maximal $S$-free set. Using separation theorems for $L$ and each of the (finitely many) points of $S$, we see that $L$ is a polyhedron. Let $m$ be the number of facets of $L$. Taking into account the finiteness of the set $S$ and using basic facts from the theory of polyhedra, one can show that the relative interior of every facet of $L$ contains a point of $S$ (this is proved rigorously in Section~\ref{sect:general}, see Lemma~\ref{S:free:sets:for:finite:S:lem}.II). Consider a sequence of points $x_1,\ldots,x_m \in S$ constructed by picking one point of $S$ from the relative interior of each facet of $L$. If $m > 2^d$, then the parity argument yields the existence of indices $i,j \in [m]$ with $i \ne j$ and $\frac{1}{2}(x_i+x_j) \in \integer^d$. Clearly, $\frac{1}{2}(x_i+x_j) \in \intr(L) \cap S$, which is a contradiction to the fact that $L$ is $S$-free. Thus, $m \le 2^d$. 

It should be mentioned that the above proof idea cannot be directly extended to the case of an arbitrary convex set $C$ by the following reason. For certain choices of $C$ there exist $d$-dimensional maximal $S$-free polyhedra $L$ such that the relative interior of some facets of $L$ does not contain points of $S$. Take, for example, $C=\setcond{x \in \real^2}{\sprod{x}{u} \le 1}$ and $L=\setcond{x \in \real^2}{\sprod{x}{u} \ge 1}$ with $u := \sizetwocol{2}{\sqrt{2}}$. Thus, in some cases using the parity argument one can give a relatively simple proof of $f( \integer^d \cap C) \le 2^d$, but for a proof in the case of a general convex set $C$ additional work is needed. Our proofs of the inequality $f( \integer^d \cap C) \le 2^d$ and its generalizations \eqref{f:mink:sum:lat:ineq} and \eqref{mixed:integer:f:ineq} will combine the parity argument with elementary tools from analysis (such as limits and the compactness argument).

\subsection*{Basic facts about $f$, $h$ and maximal $S$-free sets} \label{sect:prelim}

Let us list some simple properties of the functional $h$. The definition of the Helly number yields the equality

\begin{align}
	h(\cF) & = h(\setcond{F_1 \cap \cdots \cap F_t}{t \in \natur, \ F_1,\ldots,F_t \in \cF}) \label{helly:intersect:eq}
\end{align}
for every nonempty family $\cF$. 

In the following proposition we present several relations for $h(S)$, where $S \subseteq \real^d$.

\begin{proposition} \label{basic:h:prp}
	Let $S \subseteq \real^d$, let $C \subseteq \real^d$ be convex and let $\phi$ be an affine mapping on $\real^d$. Then
	\begin{align}
		h(S) & = \lim_{t \rightarrow +\infty} h(S \cap [-t,t]^d), \label{h:lim:ineq} \\
		h(S) & \le |S|, \label{h:card:bound:ineq} \\
		h(S \cap C) & \le h(S), \label{intersect:conv:ineq} \\
		h(\phi(S)) & \le h(S). \label{image:ineq} 
	\end{align}
\end{proposition}
\begin{proof}

	Let us prove \eqref{h:lim:ineq}. The sequence of values $h(S \cap [-t,t]^d)$ with $t \in \natur$ is monotonically nondecreasing and thus convergent, with the limit belonging to $\natur_0 \cup \{+\infty\}$. If the limit of this sequence is $+\infty$, one can see that $h(S)=+\infty$. Otherwise there exists $k \in \natur_0$ such that $h(S \cap [-t,t]^d) = k$ for all sufficiently large $t \in \natur$. Consider arbitrary convex sets $C_1,\ldots,C_m \subseteq \real^d$ with $m \in \natur$ and $C_1 \cap \cdots \cap C_m \cap S = \emptyset$. By the choice of $k$, for every sufficiently large $t \in \natur$ one can choose $i_1(t),\ldots,i_k(t) \in [m]$ such that $C_{i_1(t)} \cap \cdots \cap C_{i_k(t)} \cap S \cap [-t,t]^d = \emptyset$. By the finiteness of $[m]$ there exist indices $i_1,\ldots,i_k \in [m]$ and an infinite set $N \subseteq \natur$ such that $h(S \cap [-t,t]^d)=k$ and $i_1=i_1(t),\ldots,i_k=i_k(t)$ for every $t \in N$. The latter yields $C_{i_1} \cap \cdots \cap C_{i_k}  \cap S = \emptyset$ and shows $h(S) \le k$. 

	Inequality \eqref{h:card:bound:ineq} is trivial if $S$ is empty or infinite. Assume that $S$ is nonempty and finite, say $S=\{s_1,\ldots,s_k\}$, where $k:=|S| \in \natur$. Consider arbitrary convex sets $C_1,\ldots,C_m \subseteq \real^d$ with $m \in \natur$ and $C_1 \cap \cdots \cap C_m \cap S = \emptyset$. For every $j \in [k]$ one can choose $C_{i_j}$ with $s_j \not\in C_{i_j}$. This yields $C_{i_1} \cap \cdots \cap C_{i_k} \cap S = \emptyset$ and shows $h(S) \le |S|$. 

	Inequality \eqref{intersect:conv:ineq} follows from the fact that every $(S\cap C)$-convex set is also $S$-convex.	

	For the proof of \eqref{image:ineq} we can assume $k:=h(S) \in \natur$, since otherwise the inequality is trivial. Consider arbitrary convex sets $C_1,\ldots,C_m \subseteq \phi(\real^d)$ with $C_1 \cap \cdots \cap C_m \cap \phi(S)= \emptyset$. The latter implies $\phi^{-1}(C_1) \cap \cdots \cap \phi^{-1}(C_m) \cap S = \emptyset$. Since the sets $\phi^{-1}(C_1),\ldots,\phi^{-1}(C_m)$ are convex, there exist $i_1,\ldots,i_k \in [m]$ with $\phi^{-1}(C_{i_1}) \cap \cdots \cap \phi^{-1}(C_{i_k}) \cap S = \emptyset$. Thus $C_{i_1} \cap \cdots \cap C_{i_k} \cap \phi(S) = \emptyset$, which shows $h(\phi(S)) \le k$.	
\end{proof}

In the rest of this section we collect well-known facts about maximal $S$-free sets and the functional $f$.  If $L$ is a $d$-dimensional $S$-free polyhedron in $\real^d$ such that the relative interior of each facet of $L$ contains a point of $S$, then $L$ is maximal $S$-free.  Every $d$-dimensional $S$-free set can be extended to a maximal $S$-free set. That is, if $K$ is an $S$-free set, then there exists a maximal $S$-free set $L$ with $K \subseteq L$. This follows directly from Zorn's lemma. One can also give a different proof by adapting the approach from \cite[Proposition~3.1]{averkov-wagner-2012}. It is not hard to see that $f(S) = -\infty$ if and only if $\cl(S) = \real^d$. Furthermore, one can verify that 
\begin{equation} \label{f:prod:lin:space:eq}
	f(S \times \real^n) = f(S)
\end{equation}
for every $n \in \natur$. Equality \eqref{f:prod:lin:space:eq} follows from the fact that every $(d+n)$-dimensional $(S\times \real^n)$-free set $L$ is a subset of the $(S \times \real^n)$-free set $L' := L+ \{o\} \times \real^n$, where $L'$ can be represented as a Cartesian product of a $d$-dimensional $S$-free set and $\real^n$.

\section{Proofs of results on $f(S)$ and $h(S)$ for general sets $S$} \label{sect:general}

\subsection*{The inequality $f(S) \le h(S)$ (Theorem~\ref{main:thm}.\ref{part:f<=h}) and its consequences} \label{sect:f<=h}

In this section Theorem~\ref{main:thm}.\ref{part:f<=h} is derived as a consequence of the following lemma.

\begin{lemma} \label{key:lemma} 
	Let $S \subseteq \real^d$ and $k \in \natur$. Assume that every $d$-dimensional $S$-free  polyhedron $P$ is contained in an $S$-free polyhedron $Q$ with at most $k$ facets. Then every $d$-dimensional maximal $S$-free set is a polyhedron with at most $k$ facets.
\end{lemma}
\begin{proof} Let $L$ be an arbitrary $d$-dimensional $S$-free set. It suffices to show that $L$ is contained in an $S$-free polyhedron with at most $k$ facets.
	We consider a sequence $(P_t)_{t=1}^{+\infty}$ of $d$-dimensional polytopes such that 
	\begin{equation} \label{P_n's:nested:rel}
		P_t \subseteq P_{t+1} \quad \forall t \in \natur
	\end{equation}
	and 
	\begin{equation} \label{P_n:approximate}
		\intr(L) = \bigcup_{t=1}^{+\infty} P_t.
	\end{equation}
	Such a sequence can be constructed as follows. Let $(z_t)_{t=1}^{+\infty}$ be any ordering of the rational points of $\intr(L)$. Then, for every $t \in \natur$, we define $P_t := \conv (\{z_1,\ldots,z_{t+d}\})$. With an appropriate choice of $z_1,\ldots,z_{d+1}$ the polytope $P_1$ and, by this, also every other polytope $P_t$ is $d$-dimensional. By the assumption, each $P_t$ is contained in an $S$-free polyhedron $Q_t$ having at most $k$ facets. Every $Q_t$ can be represented by
	\[
		Q_t  = \bigsetcond{x \in \real^d}{ \sprod{u_{1,t}}{x} \le \beta_{1,t},\ldots,\sprod{u_{k,t}}{x} \le \beta_{k,t}},
	\]
	where $u_{1,t},\ldots,u_{k,t} \in \real^d$, $\beta_{1,t},\ldots,\beta_{k,t} \in \real$. In the degenerate situation $Q_t=\real^d$ we let $u_{i,t} :=o$ and $\beta_{i,t}:=1$ for every $i \in [k]$. Otherwise we can assume $u_{i,t} \ne o$ for every $i \in [k]$. After an appropriate renormalization we assume that $\sizetwocol{u_{i,t}}{\beta_{i,t}} \in \real^{d+1}$ is a vector of unit Euclidean length, for every $i \in [k]$ and $t \in \natur$. By compactness of the unit sphere in $\real^{d+1}$, there exists an infinite set $N \subseteq \natur$ such that, for every $i \in [k],$ the vector $u_{i,t}$ converges to some vector $u_i \in \real^d$ and $\beta_{i,t}$ converges to some $\beta_i \in \real$, as $t$ goes to infinity over points of $N$. Clearly, $\sizetwocol{u_{i}}{\beta_{i}} \in \real^{d+1}$ is a vector of unit length for every $i \in [k]$. We define the polyhedron 
	\[
		Q:= \bigsetcond{x \in \real^d}{ \sprod{u_1}{x} \le \beta_1,\ldots,\sprod{u_k}{x} \le \beta_k}.
	\]
	Let us show $\intr(L) \subseteq Q$. Consider an arbitrary $x' \in \intr(L)$. One has $x' \in P_t \subseteq Q_t$ for all sufficiently large $t \in N$. Thus, for each $i \in [k]$, the inequality $\sprod{u_{i,t}}{x'} \le \beta_{i,t}$ holds if $t \in N$ is sufficiently large. Passing to the limit we get $\sprod{u_i}{x'} \le \beta_i$ for every $i \in [k]$. Hence $x' \in Q$. The inclusion $\intr(L) \subseteq Q$ implies $L \subseteq Q$. It remains to show that $Q$ is $S$-free. We assume that $Q$ is not $S$-free. Then there exists $x' \in S \cap \intr(Q)$. Taking into account that $u_i\ne o$ or $\beta_i \ne 0$ for every $i \in [k]$,
	\[
		\intr(Q)  = \bigsetcond{x \in \real^d}{\sprod{u_1}{x} < \beta_1, \ldots, \sprod{u_k}{x} < \beta_k}.
	\]

	 Consequently $\sprod{u_i}{x} < \beta_i$ for all $i \in [k]$ and thus $\sprod{u_{i,t}}{x'} < \beta_{i,t}$ for all $i \in [k]$ if $t \in N$ is sufficiently large. Hence $x' \in S \cap \intr(Q_t)$ for all sufficiently large $t \in N$, contradicting the fact that $Q_t$ is $S$-free. 
\end{proof}

\begin{proof}[Proof of Theorem~\ref{main:thm}.\ref{part:f<=h}] Without loss of generality let $S \ne \emptyset$ and $k:=h(S) < +\infty$ since otherwise the assertion is trivial. The condition $S \ne \emptyset$ implies $k > 0$. Let us verify the assumption of Lemma~\ref{key:lemma}.  Let $P \subseteq \real^d$ be an arbitrary $d$-dimensional $S$-free polyhedron. We have $S \ne \emptyset$ and thus $P \ne \real^d$. We represent $P$ by $P= H_1 \cap \cdots \cap H_m$, where $m \in \natur$ and $H_1,\ldots,H_m \subseteq \real^d$ are closed halfspaces. Then $\intr(H_1) \cap S,\ldots, \intr(H_m) \cap S$ are $S$-convex sets whose intersection is empty. By the definition of $h(S)$, there exist indices $i_1,\ldots,i_k \in [m]$ such that $\intr(H_{i_1}) \cap \cdots \cap \intr(H_{i_k}) \cap S = \emptyset.$ It follows that $P \subseteq Q:= H_{i_1} \cap \cdots \cap H_{i_k}$, where $Q$ is an $S$-free polyhedron with at most $k$ facets. We conclude by  Lemma~\ref{key:lemma} that every $d$-dimensional $S$-free set is contained in a polyhedron with at most $h(S)$ facets. Hence $f(S) \le h(S)$.
\end{proof}

\begin{remark} \label{moran:dey:from:doignon:rem} \thmheader{Deriving the result of Dey \& Mor\'an from Doignon's theorem}
	Let $C$ be a convex subset of $\real^d$.
	Having obtained Theorem~\ref{main:thm}.\ref{part:f<=h}, the inequality $f(C \cap \integer^d) \le 2^d$ of Dey \& Mor\'an follows directly from Doignon's theorem $h(\integer^d)=2^d$.  By  Theorem~\ref{main:thm}.\ref{part:f<=h}, $f(C \cap \integer^d) \le h(C \cap \integer^d)$ . By convexity of $C$, we have $h(C\cap \integer^d) \le h(\integer^d)$. Taking into account Doignon's theorem, we arrive at
	$f(C \cap \integer^d) \le h(\integer^d) = 2^d$.
\end{remark}

\begin{remark} \thmheader{A weak version of Theorem~\ref{mixed:integer:f:thm}}
	Let $d \in \natur, n \in \natur_0$ and let $C \subseteq \real^{d+n}$ be convex.
	We can use the equality $h(\integer^d \times \real^n) = (n+1) 2^d$, which was proved in \cite{arXiv:1002.0948}, to find an upper bound on $f((\integer^d \times \real^n) \cap C)$. Analogously to Remark~\ref{moran:dey:from:doignon:rem}, we get $f((\integer^d \times \real^n) \cap C) \le h((\integer^d \times \real^n) \cap C) \le h(\integer^d \times \real^n) = (n+1) 2^d$. The upper bound $(n+1) 2^d$ derived in this way is weaker than the bound $2^d$ provided by Theorem~\ref{mixed:integer:f:thm}.
\end{remark}

\subsection*{The equality $f(S)=h(S)$ for discrete sets (Theorem~\ref{main:thm}.\ref{part:f=h})}

Theorem~\ref{main:thm}.\ref{part:f=h} is proved by reducing the case of a general discrete set $S$ to the case of a finite $S$. The following lemma presents properties of $S$-free sets in the case of finite $S$.

\begin{lemma} \label{S:free:sets:for:finite:S:lem}
	Let $S$ be a finite subset of $\real^d$. Then the following statements hold:
	\begin{enumerate}[I.]
		\item $f(S) = h(S)$.
		\item Every maximal $S$-free set $L$ is a $d$-dimensional polyhedron such that the relative interior of each facet of $L$ contains a point of $S$.
	\end{enumerate}
\end{lemma}
\begin{proof}
	It suffices to consider the nontrivial case $S \ne \emptyset$.
	By Theorem~\ref{main:thm}.\ref{part:f<=h}, $f(S) \le h(S)$. Thus, for showing $h(S)=f(S)$ one needs to verify $h(S) \le f(S)$. Let $A$ be an arbitrary $S$-convex set. Applying separation theorems for $\conv(A)$ and every point of $S \setminus A$, we see that $A$ can be represented as intersection of $S$ with finitely many open halfspaces. Taking into account \eqref{helly:intersect:eq}, the latter implies the equality
	$h(S) = h(\cF)$ for the family $\cF$ consisting of sets $\setcond{x \in S}{\sprod{u}{x} < \beta}$, where $u \in \real^d \setminus \{o\}$ and $\beta \in \real$. We will verify $h(\cF) \le f(S)$. Consider an arbitrary system of strict inequalities 
	\begin{align} \label{system}
		\sprod{u_i}{x} < \beta_i \qquad \forall i \in [m],
	\end{align}
	with $m \in \natur$, $u_1,\ldots,u_m \in \real^d \setminus \{o\}, \beta_1,\ldots, \beta_m \in \real$, such that \eqref{system} has no solution in $S$, that is,
	\begin{equation} \label{no:solution:in:S}
		\setcond{x \in S}{\sprod{u_i}{x} < \beta_i \ \forall i \in [m]} = \emptyset.
	\end{equation}
	For proving $h(\cF) \le f(S)$ it suffices to show that \eqref{system} has a subsystem with at most $f(S)$ inequalities which has no solution in $S$. For $j \in [m]$, we say that the $j$-th constraint of \eqref{system} is redundant if after the removal of this constraint the new system still has no solution in $S$, that is,
	\[
		\setcond{x \in S}{\sprod{u_i}{x} < \beta_i \ \forall i \in [m] \setminus \{j\}} = \emptyset.
	\]
	We say that the $j$-th constraint of the system \eqref{system}  is blocked if
	\[
		\setcond{x \in S}{\sprod{u_i}{x} < \beta_i \ \forall i \in [m] \setminus \{j\}, \ \sprod{u_j}{x}=\beta_j } \ne \emptyset.
	\]
	For every $j \in [m]$, the $j$-th constraint is either redundant or blocked by $S$ or otherwise $\beta_j < \beta_j'$, where $\beta_j' \in \real$ is given by
	\[
		\beta_j' := \min \setcond{x \in S}{\sprod{u_i}{x} < \beta_i \ \forall i \in [m] \setminus \{j\}}.
	\]
	If the $j$-th constraint is nonredundant and $\beta_j< \beta_j'$, then this constraint becomes blocked if we replace $\beta_j$ by $\beta_j'$. Note that the above operation preserves \eqref{no:solution:in:S} and, furthermore, every constraint which was previously blocked remains blocked. Removing all redundant constrains and consecutively making all the nonredundant constraints blocked we can modify every system \eqref{system} to a system 
	\begin{align} \label{fixed:system}
		\sprod{u_i}{x} < \gamma_i \qquad \forall i \in I,
	\end{align}
	with $I \subseteq [m]$ and $\gamma_i \ge \beta_i \ \forall i \in I$, such that \eqref{fixed:system} has no solution in $S$ and every constraint of \eqref{fixed:system} is blocked. The index set $I$ is nonempty since $S \ne \emptyset$. Every constraint of \eqref{fixed:system} is blocked, and so there exist  points $s_i \in S$ ($i \in I$) such that $\sprod{u_i}{s_i} = \gamma_i$ for all $i \in I$ and $\sprod{u_i}{s_j} < \gamma_i$ for all $i,j \in I$ with $i \ne j$. Consider the polyhedron $L = \setcond{x \in \real^d}{\sprod{u_i}{x} \le \gamma_i \ \forall i \in I}$. By construction, $L$ is $S$-free. Furthermore, $L$ is $d$-dimensional since for any $j \in I$ one has
	\[
		s_j - \eps u_j \in \setcond{x \in \real^d}{\sprod{u_i}{x} < \gamma_i \ \forall i \in I} = \intr(L)
	\]
	if $\eps>0$ is sufficiently small. From the properties of the points $s_i$ ($i \in I$) we conclude that $L$ has precisely $|I|$ facets and each facet of $L$ contains a point of $S$. Thus, $L$ is a $d$-dimensional maximal $S$-free set with $|I|$ facets. Hence $|I| \le f(S)$. It follows that the subsystem of \eqref{system} consisting of the inequalities indexed by $i \in I$ has no solution in $S$. This shows $h(\cF) \le f(S)$ and concludes the proof of $f(S) = h(S)$. 

	It remains to prove Part~II. Consider an arbitrary $S$-free polyhedron $P$ (not necessarily nonempty or $d$-dimensional). Then $P$ can be given by 
	\[
		P= \setcond{x \in \real^d}{\sprod{u_i}{x} \le \beta_i \ \forall i \in [m]}
	\]
	for some system \eqref{system} satisfying \eqref{no:solution:in:S}. Applying the arguments used in the proof of $f(S)=h(S)$ we see that $P$ is a subset of a $d$-dimensional maximal $S$-free polyhedron $L$ such that the relative interior of each facet of $L$ contains a point of $S$. This yields the second part of the assertion.
\end{proof}

\begin{proof}[Proof of Theorem~\ref{main:thm}.\ref{part:f=h}] Let $S \subseteq \real^d$ be discrete. In view of Lemma~ \ref{S:free:sets:for:finite:S:lem}, we restrict ourselves to the case of infinite $S$.
	The inequality $f(S) \le h(S)$ is provided by Theorem~\ref{main:thm}.\ref{part:f<=h}. It remains to show $h(S) \le f(S)$. Without loss of generality we assume $f(S)<+\infty$ since otherwise the assertion is trivial. Note that $h(S) \ge 1$ since $S$ is nonempty.

	We have $h(S) = \lim_{t \rightarrow +\infty} h(S_t)$ for $S_t := S \cap [-t,t]^d$. Thus, it suffices to show $h(S_t) \le f(S)$ for every $t \in \natur$. Let us fix an arbitrary $t \in \natur$ and let $k:=h(S_t)$. By Lemma~\ref{S:free:sets:for:finite:S:lem}.I,  $k=h(S_t)=f(S_t)$. Thus, there exists a $d$-dimensional maximal $S_t$-free polyhedron $L$ with $k$ facets. We observe that $L$ is not $S$-free in general. Let $F_1,\ldots,F_k$ be all facets of $L$.

	In the rest of the proof we proceed as follows. Using the fact that $S$ is discrete, one can deduce the existence of an $S$-free polyhedron $P$ contained in $L$ such that, for every $i \in [k]$, $P \cap F_i$ is $(d-1)$-dimensional and the relative interior of $P \cap F_i$ contains a point of $S$. It will be shown that every maximal $S$-free polytope containing $P$ has at least $k$ facets.

	By Lemma~\ref{S:free:sets:for:finite:S:lem}.II,
	\begin{equation} \label{every:relint:is:blocked:by:S_t}
		\relintr(F_i) \cap S_t \ne \emptyset \qquad \forall i \in [k].
	\end{equation}

	For each $i \in [k]$ we choose a point $x_i \in \relintr(F_i) \cap S_t$. With this choice one obtains $\conv (\{x_1,\ldots,x_k\}) \subseteq \intr(L) \cup \{x_1,\ldots,x_k\}$. Because $L$ is $S_t$-free, we get
	\[
		[-t,t]^d \cap S \cap \conv( \{x_1,\ldots,x_k\}) = S_t \cap \conv(\{x_1,\ldots,x_k\})  = \{x_1,\ldots,x_k\}.
	\]
	Since $S$ is discrete, we have
	\[
		[-t-\eps,t+\eps]^d \cap S = [-t,t]^d \cap S
	\]
	for a sufficiently small $\eps>0$. It follows that the polytope
	\[
		P := L \cap [-t-\eps,t+\eps]^d
	\]
	is $S$-free. By construction, for every $i \in [k]$, $F'_i:=F_i \cap [-t-\eps,t+\eps]^d$ is a facet of $P$ and $x_i\in \relintr(F'_i)$. There exists a maximal $S$-free set $Q$ with $P \subseteq Q$. Since we assume $f(S) <+\infty$, $Q$ is a polyhedron. For each $i \in [k]$, the $(d-1)$-dimensional polyhedron $F_i'$ is a subset of a facet of $Q$, since otherwise the point $x_i \in \relintr(F'_i)$ would be in the interior of $Q$. Let $i,j \in [k]$ and $i \ne j$. Since $F_i$ and $F_j$ are distinct facets of $L$, it follows that the sets $F_i'$ and $F_j'$ are subsets of distinct facets of $Q$. Hence $Q$ has at least $k$ facets. This yields $h(S_t) \le f(S)$ for every $t \in \natur$ and implies $h(S) \le f(S)$.
\end{proof}

\subsection*{The $\liminf$ theorem (Theorem~\ref{liminf:thm}) and its consequences} \label{sect:proof:IV}

\begin{proof}[Proof of Theorem~\ref{liminf:thm}]
	We first show \eqref{h:liminf:ineq}. Let $k:=\liminf_{t \rightarrow +\infty} h(S_t)$. It suffices to consider the case $0<k<+\infty$. There exists an infinite set $N \subseteq \natur$ such that $h(S_t)=k$ for every $t \in N$.

	Consider arbitrary convex sets $C_1,\ldots,C_m$ in $\real^d$ with $m \in \natur$ and assume that $C_1 \cap \cdots \cap C_m \cap S= \emptyset$. We show the existence of indices $i_1,\ldots,i_k \in [m]$ satisfying $C_{i_1} \cap \cdots \cap C_{i_k} \cap S = \emptyset$. By the definition of the Helly number for every $t \in N$ there exist $i_1(t),\ldots,i_k(t) \in [m]$ with $C_{i_1(t)} \cap \cdots \cap C_{i_k(t)} \cap S_t = \emptyset$. By the finiteness of $[m]$, there exist $i_1,\ldots,i_k \in [m]$ such that $i_1=i_1(t),\ldots,i_k=i_k(t)$ for all $t \in N'$, where $N'$ is an infinite subset of $N$. The monotonicity property $S_t \subseteq S_{t+1} \ \forall  t \in \natur$ implies $S = \bigcup_{t \in N'} S_t$. Hence, $C_{i_1} \cap \cdots \cap C_{i_k} \cap S = \emptyset$.

	We now show \eqref{f:liminf:ineq}. Let $k:= \liminf_{t \rightarrow +\infty} f(S_t)$. We restrict ourselves to the nontrivial case $k<+\infty$.  There exists an infinite set $M \subseteq \natur$ such that $f(S_t) = k$ for every $t \in M$. In the degenerate cases one can have $k=-\infty$ or $k=0$. If $f(S_t)=-\infty$, then $\cl(S_t)=\real^d$ and, consequently, $\cl(S) = \real^d$. This yields $f(S) = -\infty$. In the case $k=0$ the assertion is trivial since $S_t = \emptyset \ \forall t \in M$ and thus $S=\emptyset$. Let $k \in \natur$. We consider an arbitrary $d$-dimensional $S$-free set $L \subseteq \real^d$ and show that $L$ is a subset of an $S$-free polyhedron with at most $k$ facets. For every $t \in M$ there exists a maximal $S_t$-free set $Q_t$ with $L \subseteq Q_t$. Since $f(S_t)=k$, $Q_t$ is a polyhedron with at most $k$ facets. We proceed similarly to the proof of Lemma~\ref{key:lemma}. For every $t \in M$, $Q_t$ can be given by 
	\[
		Q_t = \bigsetcond{x \in \real^d}{\sprod{u_{1,t}}{x} \le \beta_{1,t},\ldots,\sprod{u_{k,t}}{x} \le \beta_{k,t}},
	\]
	where $u_{1,t},\ldots,u_{k,t}\in \real^d$, $\beta_{1,t},\ldots,\beta_{k,t} \in \real$ and, for every $i \in [k]$ and $t \in M$,  $\sizetwocol{u_{i,t}}{\beta_{i,t}} \in \real^{d+1}$ is a vector of unit Euclidean length.  By compactness, there exists an infinite $N \subseteq M$ such that, for every $i \in [k]$, $u_{i,t}$ converges to some $u_i \in \real^d$ and  $\beta_{i,t}$ converges to some $\beta_i \in \real$ as  $t$ goes to infinity over points of $N$. Consider the polyhedron
	\[
		Q = \bigsetcond{x \in \real^d}{\sprod{u_1}{x} \le \beta_1,\ldots,\sprod{u_k}{x} \le \beta_k}.
	\]
	Let us verify the inclusion $L \subseteq Q$. The inclusions $L \subseteq Q_t \ \forall t \in \natur$ imply $\sprod{x}{u_{i,t}} \le \beta_{i,t}$ for all $x \in L, i \in [k], t \in \natur$. Passing to the limit, as $t$ goes to infinity over points of $N$, we obtain $\sprod{x}{u_i} \le \beta_i$ for every $i \in [k]$. Hence $L \subseteq Q$. Let us show that $Q$ is $S$-free. For this it suffices to show that, for every $t \in \natur,$ $Q$ is $S_t$-free. Assume that $Q$ is not $S_t$-free.  Then there exists a point $x' \in S_t$ lying in $\intr(Q)$, that is, satisfying $\sprod{u_i}{x'}  < \beta_i$ for every $i \in [k].$ This points satisfies $\sprod{u_{i,t'}}{x'} < \beta_{i,t'}$ for every $i \in [k]$, where $t' \in N$, $t' \ge t$ is sufficiently large. We have $S_t \subseteq S_{t'}$, and so $x' \in S_{t'} \cap \intr(Q_{t'})$, which contradicts the fact that $Q_{t'}$ is $S_{t'}$-free.
\end{proof}

\begin{remark} \label{indep:proofs:rem} \thmheader{Alternative proofs of results of Doignon and Lov\'asz}  By Theorem~\ref{main:thm}.\ref{part:f=h}, $h(\integer^d)=f(\integer^d)$. This shows that the result $h(\integer^d) =2^d$ of Doignon and the result $f(\integer^d) = 2^d$ of Lov\'asz are equivalent. Let us give self-contained proofs of these two results. Let $k= f(\integer^d) = h(\integer^d)$.  The family of the $\integer^d$-convex sets $\{0,1\}^d \setminus \{z\}$ with $z \in \{0,1\}^d$ contains $2^d$ elements and has empty intersection. Every proper subfamily of this family has nonempty intersection. Hence $k=h(\integer^d) \ge 2^d$. For deriving $k=f(\integer^d) \le 2^d$ we first apply \eqref{f:liminf:ineq}, which yields $k \le \liminf_{t \rightarrow +\infty} f(S_t)$, where $S_t:= \integer^d \cap [-t,t]^d \ \forall t \in \natur$. Since $[-t,t]^d$ is bounded and convex, we get $f(S_t) \le 2^d$ by the parity argument, as explained in Section~\ref{sect:parity:arg}. Thus, $k \le 2^d$.

Note that, in contrast to the above approach, the existing proofs of $f(\integer^d)=2^d$ use tools from the geometry of numbers (see the sources \cite{arXiv:1110.1014} and \cite{MR2724071}, which use Minkowski's first fundamental theorem and Diophantine approximation, respectively).
\end{remark}

\section{Proofs of upper bounds on $f(S)$ and $h(S)$ for \\ structured sets $S$} \label{sect:structured}

\subsection*{Proof of the theorem on adding a convex set (Theorem~\ref{part:mink:sum})}

\begin{proof}[Proof of Theorem~\ref{part:mink:sum}]
	Let us show \eqref{h:mink:sum:ineq}. Let $n:=\dim(C)$. Consider an appropriate injective affine mapping $\psi : \real^n \rightarrow \real^d$ and an $n$-dimensional convex set $C' \subseteq \real^n$ such that $\psi(C') = C$. We introduce the mapping $\phi : \real^{d+n} \rightarrow \real^d$ by $\phi \sizetwocol{x}{y} = x+ \psi(y)$ for every $x \in \real^d$ and $y \in \real^n$. By construction, $\phi(S \times C' ) = S+C$. Applying
	the inequality
	\[
		h(S \times \real^n) \le (n+1) h(S),
	\]
	proven in \cite{arXiv:1002.0948}, together with \eqref{intersect:conv:ineq} and \eqref{image:ineq}, we obtain
	\begin{align*}
		h(S+C) & = h(\phi(S \times C')) \le h(S \times C') = h((S \times \real^n) \cap (\real^d \times C')) \\ & \le h(S \times \real^n) \le (n+1) h(S).
	\end{align*}

	Now we derive \eqref{f:mink:sum:ineq}. Assume $k:=f(S) < +\infty$, since otherwise the assertion is trivial. Consider an arbitrary $d$-dimensional $(C+S)$-free set $L \subseteq \real^d$, that is, $L$ is $d$-dimensional, closed, convex and $\intr(L) \cap (C+S) = \emptyset$. It suffices to show that $L$ is contained in a maximal $(C+S)$-free polyhedron with at most $k$ facets. 	Let us first show the equality
	\begin{equation}
		\intr(L) - C = \intr(L-C) \label{aux:eq}
	\end{equation}
	The left hand side is an open set, since it can be represented as the union of the open sets $\intr(L)-c$ with $c \in C$. Hence $\intr(L)-C=\intr(\intr(L)-C) \subseteq \intr(L-C)$. Let us show the converse inclusion $\intr(L-C) \subseteq \intr(L) -C$. Using the additivity of the operator $\relint$ (see \cite[\S6]{MR0274683}) and the $d$-dimensionality of $L$ we obtain $\intr(L-C) = \intr(L) - \relintr(C) \subseteq \intr(L) - C.$

	In view of \eqref{aux:eq}, we get
	\begin{align*}
		\intr(L)  \cap (C+S) & = \emptyset & &\Longleftrightarrow & (\intr(L) - C) \cap S &  = \emptyset &   &\Longleftrightarrow   & \intr(L-C) \cap S & = \emptyset.
	\end{align*}

	Thus, $\cl(L-C)$ is $S$-free and by this there exists a $d$-dimensional $S$-free polyhedron $P$ with at most $k$ facets and $L-C \subseteq P$. The polyhedron $P$ can be given by 
	\[
		P= \setcond{x \in \real^d}{\sprod{u_1}{x} \le \beta_1,\ldots,\sprod{u_k}{\beta_k} \le \beta_k},
	\]
	where $u_1,\ldots,u_k \in \real^d \setminus \{o\}$ and $\beta_1,\ldots,\beta_k \in \real$. One has 
	$\sprod{x  - c}{u_i} \le \beta_i$ for all $x \in L, c \in C$, $i \in [k]$. Hence
	\[
		L \subseteq P':= \setcond{x \in \real^d}{\sprod{u_1}{x} \le \beta_1',\ldots,\sprod{u_k}{x} \le \beta_k'}
	\]
	 where 
	\[
		\beta_i':=\beta_i + \inf_{c \in C} \sprod{u_i}{c} \qquad \forall i \in [k].
	\]
	It remains to show that $P'$ is $(C+S)$-free. Assume the contrary. Then one can find $x' \in S$ and $c' \in C$ such that $\sprod{u_i}{x' + c'} < \beta_i + \inf_{c \in C} \sprod{u_i}{c}$ for every $i \in [k]$. It follows that $\sprod{u_i}{x'} < \beta_i$ for every $i \in [k]$, and hence $P$ is not $S$-free, which is in contradiction to the choice of $P$.
\end{proof}

\subsection*{Proofs of Theorems~\ref{part:lattices} and \ref{mixed:integer:f:thm}} \label{sect:lattices}

In the proof of Theorem~\ref{part:lattices} we shall need properties of $d$-dimensional maximal $S$-free sets in the case that $S$ is a union of finitely many polyhedra. Such maximal $S$-free sets will be discussed in Lemma~\ref{polyhedral:S:lem}. For the proof of Lemma~\ref{polyhedral:S:lem} we shall need a result on separation of polyhedra, which we formulate in Lemma~\ref{fine:separation:polyh:lem}. Lemma~\ref{S:between:lem} is another auxiliary result used in the proof of Theorem~\ref{part:lattices}.

If $P \subseteq \real^d$ is a polyhedron, then every nonempty face $F$ of $P$ can be given as 
\[
	F= F(P,u) := \setcond{x \in P}{\sprod{u}{y} \le \sprod{u}{x}  \ \forall y \in P}
\]
for an appropriate $u \in \real^d$.  The notation $F(P,u)$ is used in Lemma~\ref{fine:separation:polyh:lem}.

\begin{lemma} \label{fine:separation:polyh:lem} 
	Let $P, Q \subseteq \real^d$ be polyhedra such that $\dim(P)=d$, $Q \ne \emptyset$ and $\intr(P) \cap Q = \emptyset$. Assume that for some facet $G$ of $P$ one has $\relintr (G) \cap Q   = \emptyset$. Then there exists a closed halfspace $H \subseteq \real^d$ such that $P \subseteq H$, $\relintr (G) \subseteq \intr (H)$, and $Q \cap \intr (H) = \emptyset$.
\end{lemma}
\begin{proof}
	It is known that the set $P-Q$ is a polyhedron. 

	\emph{Case~1: $o \not\in P-Q$}. We can separate $o$ and $P-Q$ by a hyperplane. That is, there exists $u \in \real^d \setminus \{o\}$ such that $\sup_{x \in P-Q} \sprod{u}{x} <0$. It follows that $\sup_{x \in P} \sprod{u}{x} - \inf_{y \in Q} \sprod{u}{y} < 0$. We can define $H = \setcond{x \in \real^d}{\sprod{u}{x} \le \beta}$, where $\beta \in \real$ is any value satisfying $\sup_{x \in P} \sprod{u}{x} < \beta < \inf_{y \in Q} \sprod{u}{y}$. For $H$ as above one has $P \subseteq \intr(H)$ and $H \cap Q = \emptyset$.

	\emph{Case~2: $o \in P-Q$}. The set $P-Q$ is $d$-dimensional since $\dim (P) = d$ and $Q \ne \emptyset$. Observe that $o \in \bd(P-Q)$. In fact, otherwise we would have $o \in \intr(P-Q) = \intr(P) - \relintr(Q)$, which implies $\intr(P) \cap Q \ne \emptyset$ yielding a contradiction. Since $o \in \bd(P-Q)$, there exists $u \in \real^d \setminus \{o\}$ such that  $F(P-Q,u)$ is a face of $P-Q$ with $o \in \relintr (F(P-Q,u))$. It is known that $F(P-Q,u)=F(P,u)-F(Q,u)$. Taking into account the additivity of the $\relintr$ operator we obtain $\relintr(F(P-Q,u))=\relintr(F(P,u)-F(Q,u)) = \relintr(F(P,u)) - \relintr(F(Q,u))$. One has $F(P,u) \ne G$,  since otherwise $o \in \relintr(G) - Q$, and by this $\relintr(G) \cap Q \ne \emptyset$, which contradicts the assumptions. Therefore, we have $F(P,u) \cap \relintr(G) = \emptyset$. Let $\beta :=\sup_{x \in P} \sprod{u}{x} = \inf_{y \in Q} \sprod{u}{y}$ and $H := \setcond{x \in \real^d}{\sprod{u}{x} \le \beta}$. The assertion for $H$ can be verified in a straightforward manner.
\end{proof}

We also observe that Lemma~\ref{fine:separation:polyh:lem} can be proved algebraically using Motzkin's transposition theorem (a generalization of Farkas' lemma to systems involving strict as well as nonstrict linear inequalities; see \cite[Corollary~7.1k]{MR874114})

\begin{lemma} \label{polyhedral:S:lem}
	Let $d, k \in \natur$ and let $S := C_1 \cup \cdots \cup C_k$, where $C_1,\ldots,C_k \subseteq \real^d$ are nonempty polyhedra. Let $L$ be a $d$-dimensional maximal $S$-free set in $\real^d$. Then $L$ is a polyhedron and the relative interior of each facet of $L$ contains a point of $S$.
\end{lemma}
\begin{proof}
	For each $i \in [k]$, $L$ and $C_i$ can be separated by a hyperplane. In view of the maximality of $L$, it follows that $L$ is a polyhedron with at most $k$ facets. We consider an arbitrary facet $F$ of $L$ and show that $\relintr(F) \cap S \ne \emptyset$ arguing by contradiction. Assume that $\relintr (F) \cap S = \emptyset$, that is, $\relintr (F) \cap C_i = \emptyset$ for every $i \in [k]$. By Lemma~\ref{fine:separation:polyh:lem}, for every $i \in [k]$ there exists a closed halfspace $H_i$ in $\real^d$ such that $L \subseteq H_i$, $\relintr (F) \subseteq \intr (H_i)$ and $\intr (H_i) \cap C_i = \emptyset$. We define $P:= H_1 \cap \cdots \cap H_k$. By construction, $P$ is $S$-free. From the construction we also get $L \subseteq P$ and $\relintr (F) \subseteq \intr (P)$, which yields $L \varsubsetneq P$. This contradicts the maximality of the $S$-free set $L$.
\end{proof}

\begin{lemma} \label{S:between:lem}
	Let $(C_t)_{t=1}^{+\infty}$ be a sequence of convex sets in $\real^d$ and let $S=\bigcup_{t=1}^{+\infty} C_t$. Let $S' \subseteq \real^d$ be a set satisfying
	\[
		\bigcup_{t=1}^{+\infty} \relintr(C_t) \subseteq S' \subseteq \bigcup_{t=1}^{+\infty} \cl(C_t).
	\]
	Then $f(S)=f(S').$
\end{lemma}
\begin{proof}
	Let $L \subseteq \real^d$ be a $d$-dimensional closed, convex set. We show that $L$ is $S$-free if and only if $L$ is $S'$-free. One has
	\begin{align*}
		S \cap \intr(L) &= \emptyset & &\Longrightarrow & &C_t \cap \intr(L) = \emptyset  & &\forall t \in \natur \\
		& & &\Longrightarrow & &C_t  \subseteq \real^d \setminus \intr(L) & &\forall t \in \natur \\
		& & &\Longrightarrow & &\cl(C_t) \subseteq \real^d \setminus \intr(L) & &\forall t \in \natur \\
		& & &\Longrightarrow & &S' \subseteq \real^d \setminus \intr(L) & & \\
		& & &\Longrightarrow & & S' \cap \intr(L) = \emptyset. & &
	\end{align*}
	Conversely, 
	\begin{align*}
		S' \cap \intr(L) & = \emptyset & &\Longrightarrow & & \relintr(C_t) \cap \intr(L) = \emptyset & &\forall t \in \natur\\
		& & &\Longrightarrow & &\relintr(C_t) \subseteq \real^d \setminus \ \intr(L) & &\forall t \in \natur \\
		& & &\Longrightarrow & &\cl(\relintr(C_t)) \subseteq \real^d \setminus \intr(L) & &\forall t \in \natur \\
		& & &\Longrightarrow & & \cl(C_t) \subseteq \real^d \setminus \intr(L) & &\forall t \in \natur \\
		& & &\Longrightarrow & &S \subseteq \real^d \setminus \intr(L) & & \\
		& & &\Longrightarrow & &S \cap \intr(L) = \emptyset.
	\end{align*}
	Thus, we get the assertion.
\end{proof}

\begin{proof}[Proof of Theorem~\ref{part:lattices}]
	Let us prove \eqref{h:mink:sum:lat:ineq}. Since $D$ is convex we obtain $h((C+\Lambda) \cap D)) \le h(C+\Lambda)$. By \eqref{h:mink:sum:ineq} we have $h(C+\Lambda) \le (\dim(C)+1) h(\Lambda)$. Doignon's theorem yields $h(\Lambda) = 2^{\rank(\Lambda)}$. This implies \eqref{h:mink:sum:lat:ineq}.

	Now let us show \eqref{f:mink:sum:lat:ineq}. Let $S:= (C + \Lambda) \cap D$ and let $L \subseteq \real^d$ be a $d$-dimensional maximal $S$-free set. We prove that $L$ is a polyhedron with at most $2^{\rank(\Lambda)}$ facets. Without loss of generality let $S \ne \emptyset$. Consider the following two cases.

	\emph{Case~1: $C$ and $D$ are both polytopes.} In this case $S$ is a union of finitely many polytopes. By Lemma~\ref{polyhedral:S:lem} the set $L$ is a polyhedron. Let $F_1,\ldots,F_m$ be all facets of $L$, where $m \in \natur_0$. By Lemma~\ref{polyhedral:S:lem}, for every $i \in [m]$, the relative interior of $F_i$ contains a point of the form $x_i+y_i$, where $x_i \in \Lambda$, $y_i \in C$ and $x_i+y_i \in D$. If $m> 2^{\rank(\Lambda)}$, then by the parity argument (see Section~\ref{sect:parity:arg}) one has $\frac{1}{2}(x_i+x_j) \in \Lambda$ for some $1 \le i < j \le m$. Using the convexity of $C, D$ and $L$ and the fact that $x_i+y_i$ and $x_j+y_j$ lie in the relative interior of two different facets of $L$ we deduce $\frac{1}{2}((x_i+y_i)+(x_j+y_j)) \in \intr (L) \cap S$. Hence $L$ is not $S$-free, which is a contradiction. It follows that $m \le 2^{\rank(\Lambda)}$.

	\emph{Case~2: $C$ or $D$ is not a polytope.} We introduce a sequence $(z_i)_{i=1}^{+\infty}$ such that $\setcond{z_i}{i \in \natur} = \setcond{z \in \Lambda}{(z+C) \cap D \ne \emptyset}$. Then
	\[
		S = \bigcup_{i=1}^{\infty} Q_i,
	\]
	where $Q_i := (z_i+C) \cap D \ne \emptyset$.
	For each $Q_i$ we fix a countable, dense subset $\setcond{ q_{i,j} }{j \in \natur}$ of $\relint(Q_i)$. With this choice one has $\relintr (Q_i) = \conv(\setcond{q_{i,j}}{j \in \natur})$. For every $t \in \natur$ we define
	\[
		S_t:= (\Lambda + C_t ) \cap D_t,
	\]
	where
	\begin{align*}
		C_t & := \conv (\setcond{q_{i,j} - z_i}{i,j \in [t]}), \\
		D_t & := \conv (\setcond{q_{i,j}}{i,j \in [t]}).
	\end{align*}

	By construction, one has $C_t \subseteq C_{t+1} \subseteq C$ and $D_t \subseteq D_{t+1} \subseteq D$ for every $t \in \natur$. The latter also implies $S_t \subseteq S_{t+1} \subseteq S$ for every $t \in \natur$. We have
	\begin{align*}
		\bigcup_{i=1}^{\infty} \relint(Q_i) & = \bigcup_{i=1}^{\infty} \conv ( \setcond{q_{i,j}}{j \in \natur}) = \bigcup_{i=1}^{\infty} \bigcup_{t=i}^{\infty} \conv ( \setcond{q_{i,j}}{j \in [t]}) \\ & = \bigcup_{t=1}^{\infty} \bigcup_{i=1}^t \conv ( \setcond{q_{i,j}}{j \in [t]}) \subseteq \bigcup_{t=1}^{\infty} (\Lambda + C_t) \cap D_t = \bigcup_{t=1}^{\infty} S_t \subseteq S,
	\end{align*}
	where above we use the equality
	\begin{equation} \label{caratheodory:type:eq}
		\conv( \setcond{q_{i,j}}{j \in \natur}) = \bigcup_{t=i}^{\infty} \conv( \setcond{q_{i,j}}{j \in [t]}) \qquad \forall i \in \natur.
	\end{equation}
 	Note that \eqref{caratheodory:type:eq} follows directly from Carath\'eodory's theorem (see \cite[Theorem~1.1.4]{MR1216521}). 

	Let $S' := \bigcup_{t=1}^{\infty} S_t$. By Lemma~\ref{S:between:lem}, one has $f(S')=f(S)$.
	By  Theorem~\ref{liminf:thm}, $f(S') \le \liminf_{t \rightarrow +\infty} f(S_t)$. Applying the assertion obtained in Case~1 we get $f(S_t) \le 2^{\rank(\Lambda)}$ for every $t \in \natur$. The latter yields $f(S) = f(S') \le 2^{\rank(\Lambda)}$.
\end{proof}

\begin{proof}[Proof of Theorem~\ref{mixed:integer:f:thm}]
	Observing that $\integer^d \times \real^n = \integer^d \times \{o\} + \{o\} \times \real^n$, where $\integer^d \times \{o\}$ is a lattice and $\{o\} \times \real^n$ is a linear space (and thus, a convex set), we see that \eqref{mixed:integer:f:ineq} is a direct consequence of \eqref{f:mink:sum:lat:ineq}. The equality case $f(\integer^d \times \real^n) = 2^d$ follows from \eqref{f:prod:lin:space:eq} and $f(\integer^d) = 2^d$.
\end{proof}

\section*{Acknowledgements}

I thankfully acknowledge the help of the anonymuous referees in improving the structure of the manu\-script and presentation of some of the arguments (most notably, the compactness argument in the proofs of Lemma~\ref{key:lemma} and Theorem~\ref{liminf:thm}).  Moreover, the referees pointed out a mistake in the proof of  \eqref{mixed:integer:f:ineq} in an earlier version of this paper. This indication has been a source of motivation for my further research on the topic of this paper. As a result, the text of the manuscript has been extended with Theorems~\ref{main:thm}.\ref{part:f=h}, \ref{liminf:thm}, \ref{part:mink:sum}, and \ref{part:lattices}.


\providecommand{\bysame}{\leavevmode\hbox to3em{\hrulefill}\thinspace}
\providecommand{\MR}{\relax\ifhmode\unskip\space\fi MR }
\providecommand{\MRhref}[2]{%
  \href{http://www.ams.org/mathscinet-getitem?mr=#1}{#2}
}
\providecommand{\href}[2]{#2}

\end{document}